\documentclass[11pt]{amsart}
\usepackage[margin=1.25in]{geometry}
\usepackage{amsmath, amsthm, amssymb}
\usepackage{mathtools, graphicx}
\usepackage{hyperref}
\usepackage{cleveref}
\usepackage{amsaddr}
\usepackage{multirow}
\usepackage{algorithm, algpseudocode}

\usepackage{xcolor}

\newtheorem{THM}{Theorem}
\newtheorem{LMA}[THM]{Lemma}
\newtheorem*{DEF}{Definition}

\crefname{THM}{theorem}{theorems}
\crefname{LMA}{lemma}{lemmas}
\crefname{equation}{equation}{equation}

\numberwithin{THM}{section}
\numberwithin{equation}{section}

\newcommand{\RR}{\mathbb{R}}
\newcommand\ov[1]{\overline{#1}}
\newcommand{\lb}{\lbrack}
\newcommand{\rb}{\rbrack}

\newcommand\wt[1]{{\widetilde{#1}}}

\newcommand\half{\frac{1}{2}}

\DeclarePairedDelimiter{\floor}{\lfloor}{\rfloor}
\newcommand\restr[2]{{
  \left.\kern-\nulldelimiterspace
  #1
  \vphantom{\big|}
  \right|_{#2}
  }}

\begin{document}
\title[Quasi-Newton Methods without Line Searches]{Quasi-Newton methods: Superlinear Convergence Without Line Searches for Self-Concordant Functions}
\author{Wenbo Gao$^{\ast\dagger}$ and Donald Goldfarb$^\dagger$}
\address{Industrial Engineering and Operations Research, Columbia University \\ 500 W 120th St., New York, NY 10027}
\date{June 03, 2018.}
\thanks{$^\ast$Corresponding author.}
\thanks{$^\dagger$Research of this author was supported in part by NSF Grant CCF-1527809.}
\email{\texttt{wg2279@columbia.edu, goldfarb@columbia.edu}}
\begin{abstract}
We \textcolor{black}{consider the use of} a \emph{curvature-adaptive} step size \textcolor{black}{in} gradient-based iterative methods, including quasi-Newton methods, for minimizing self-concordant functions, extending an approach first proposed for Newton's method by Nesterov. This step size has a simple expression that can be computed analytically; hence, line searches are not needed. We show that using this step size in the BFGS method (and quasi-Newton methods in the Broyden convex class other than the DFP method) results in superlinear convergence for strongly convex self-concordant functions. We present numerical experiments comparing gradient descent and BFGS methods using the curvature-adaptive step size to traditional methods on \textcolor{black}{deterministic} logistic regression problems, \textcolor{black}{and to versions of stochastic gradient descent on stochastic optimization problems.}
\end{abstract}
\maketitle

\section{Introduction}
We are concerned in this paper with iterative optimization algorithms, which at each step, first select a \emph{direction} $d_k$ and then determine a \emph{step size} $t_k$. Such algorithms, which are usually referred to as \emph{line search} algorithms, need to choose an appropriate step size $t_k$ to perform well, both in theory and in practice.

Theoretical proofs of global convergence generally assume one of the following approaches for selecting the step sizes:
\begin{enumerate}
	\item{The step sizes are obtained from line searches.}
	\item{The step size is a constant, often chosen `sufficiently small'.}
\end{enumerate}
Inexact line searches, and in particular those that choose steps that satisfy the Armijo-Wolfe conditions, or just the latter combined with backtracking, are usually performed and work well in practice. However, they can be costly to perform, and are often prohibitively costly for many common objective functions such as those that arise in machine learning, computer vision, and natural language processing. Moreover, in stochastic optimization algorithms, line searches based on stochastic function values and gradients, which are only estimates of the true quantities (see \Cref{further}), can be meaningless. In contrast, constant step sizes $t_k = t$ for all $k$ require no additional computation beyond selecting $t$, but determining an appropriate constant $t$ may be difficult. The value of $t$ required in the theoretical analysis is often too small for practical purposes, and moreover, is impossible to compute without knowledge of unknown parameters (e.g. the Lipschitz constant of $\nabla f$). A single constant step size may also be highly suboptimal, as the iterates transition between regions with different curvature.

The basic idea for a step size determined by the local curvature of the objective function $f$ was developed by Nesterov, who introduced the \emph{damped Newton method} \cite{N_ICP}. This idea is closely related to a well-behaved class of functions known as \emph{self-concordant functions} \cite{NN}, which we define in \Cref{scf_basic}. When applied to a self-concordant function $f$, the damped Newton method is globally convergent and locally converges quadratically. These results were extended in recent work.
\begin{enumerate}
	\item Tran-Dinh et al. \cite{TRANDINH} propose a 
	proximal framework for composite self-concordant minimization, which includes proximal damped Newton, proximal quasi-Newton, and proximal gradient descent. They establish that proximal damped Newton is globally convergent and locally quadratically convergent, and that proximal damped gradient descent is globally convergent and locally linearly convergent. However, they do not propose a proximal quasi-Newton algorithm or prove global convergence for a generic version of such an algorithm.
	\item Zhang and Xiao \cite{ZX} propose a distributed method for self-concordant empirical loss functions, based on the damped Newton method, and establish its convergence.
	\item Lu \cite{LZS} proposes a randomized block proximal damped Newton method for composite self-concordant minimization, and establishes its convergence.
\end{enumerate}

While the damped Newton method has been extensively studied, no comparable theory exists for quasi-Newton methods in the self-concordant setting. It is well known that for convex functions, proving global convergence for the BFGS method \cite{Broyden,Fletcher,G1970,shanno1970conditioning} with inexact line searches is far more challenging than proving global convergence for scaled gradient methods, and that a similar statement holds for the $Q$-superlinear convergence of the BFGS method applied to strongly convex functions compared with, for example, proving $Q$-quadratic convergence of Newton's method. With regard to $Q$-superlinear convergence, it is well known \cite{POW} that if the the largest eigenvalue of the Hessian of the objective is bounded above, and if the sum of the distances of the iterates generated by the BFGS method from the global minimizer is finite, then the BFGS method converges $Q$-superlinearly. We note that Tran-Dinh et al. \cite{TRANDINH} give a proof of this local result for their ``pure"-proximal-BFGS method (i.e., one that uses a step size of $1$ on every iteration and starts from a point ``close" to the global minimizer), but they do not prove that this method generates iterates satisfying the required conditions. This leaves open the question of how to design a globally convergent ``dampled'' version of the BFGS method for self-concordant functions. In particular, we wish to avoid assuming either the Dennis-Mor\'{e} condition \cite{DM1} or the summability of the distances to the global minmizer, since these conditions are extremely strong, verging on being tautological, as assumptions.

In this paper 
we extend the theory of self-concordant minimization developed by Nesterov and Nemirovski \cite{NN}
and further developed by 
Tran-Dinh et al. \cite{TRANDINH}.
Our focus here is mainly on filling the gap in this theory for quasi-Newton methods. To simplify the presentation, we consider only quasi-Newton methods that use the BFGS update, although our results apply to all methods in the Broyden class of quasi-Newton methods other than the DFP method \cite{davidon59,fp63}. We introduce a framework for non-composite optimization; i.e., we do not consider proximal methods as in \cite{TRANDINH}. The key feature of this framework is a step size that is optimal with respect to an upper bound on the decrease in the objective value, which we call the \emph{curvature-adaptive step size}. We use the term curvature-adaptive, or simply \emph{adaptive}, to refer to this step size choice or to methods that employ it, so as not to confuse such methods with damped BFGS updating methods (e.g., see \cite[\S18.3]{NW}), which are unrelated.

We first prove that scaled gradient methods that use the curvature-adaptive step size are globally $R$-linearly convergent on strongly convex self-concordant functions. We note that in \cite{TRANDINH}, this step size is also identified, but that the $R$-linear convergence is only proved locally. We then prove our main result, on quasi-Newton methods: that the BFGS method, using this step size, is globally convergent for functions that are self-concordant, bounded below, and have a bounded Hessian, and furthermore, is $Q$-superlinearly convergent when the function is strongly convex and self-concordant. For completeness, we then present several numerical experiments which shed insight on the behavior of adaptive methods. \textcolor{black}{These show that for deterministic optimization, using curvature-based step sizes in quasi-Newton methods is dominated by using inexact line searches, whereas in stochastic settings, using curvature-based step sizes is very helpful compared to constant step sizes.}

Our paper is organized as follows. In \Cref{prelim}, we introduce the notation and assumptions that we use throughout the paper. In \Cref{scf_basic}, we define the class of self-concordant functions and describe their essential properties. In \Cref{adapted}, we introduce our 
framework for self-concordant minimization and \textcolor{black}{provide a derivation of} what we call the \emph{curvature-adaptive} step size\textcolor{black}{, which extends the curvature-based step size obtained in \cite{TRANDINH} for proximal gradient methods}. In \Cref{boundedH}, we apply our approach to scaled gradient methods, and give a simple proof that these methods are globally $R$-linearly convergent on strongly convex self-concordant functions. In \Cref{BFGS}, we present our main results. Specifically, we prove there that the BFGS method with curvature-adaptive step sizes is globally and $Q$-superlinearly convergent. In \Cref{numerical}, we present numerical experiments testing our new methods on logistic regression problems in the deterministic setting. In \Cref{further}, we discuss stochastic extensions of adaptive methods. \textcolor{black}{In \Cref{apdx:stoch}, we provide a numerical example of solving an online stochastic problem using stochastic adaptive methods.}

\section{Preliminaries}\label{prelim}
We use $f: \RR^n \rightarrow \RR$ to denote the objective function, and $g(\cdot), G(\cdot)$ denote the gradient $\nabla f(\cdot)$ and Hessian $\nabla^2 f(\cdot)$, respectively. In the context of a sequence of points $\{x_k\}_{k=0}^\infty$, we write $g_k$ for $g(x_k)$ and $G_k$ for $G(x_k)$. Unless stated otherwise, the function $f$ is assumed to have continuous third derivatives (as $f$ is generally assumed to be self-concordant), which we write as $f \in \mathcal{C}^3$.

The norm $\|\cdot\|$ denotes the 2-norm, and when applied to a matrix, the operator 2-norm. 

\section{Self-Concordant Functions}\label{scf_basic}
The notion of \emph{self-concordant} functions was first introduced by Nesterov and Nemirovski \cite{NN} for their analysis of Newton's method in the context of interior-point methods. Nesterov \cite{N_ICP} provides a clear exposition and motivates self-concordancy by observing that, while Newton's method is invariant under affine transformations, the convergence analysis makes use of norms which are \emph{not} invariant. To remedy this, Nesterov and Nemirovski replace the Euclidean norm by an invariant local norm, and replace the assumption of Lipschitz continuity of the Hessian $G(x)$ by the self-concordancy of $f$.
\begin{DEF}
	Let $f$ be a convex function. The \emph{local norm} of $h \in \RR^n$ at a point $x$ where $G(x) \succ 0$ is given by
	$$ \|h\|_x = \sqrt{ h^TG(x)h}.$$
\end{DEF}
\begin{DEF}
	A closed convex function $f: \RR^n \rightarrow \RR$ is \emph{self-concordant} if $f \in \mathcal{C}^3$ and there exists a constant $\kappa$ such that for every $x \in \RR^n$ and every $h \in \RR^n$, we have
	$$|\nabla^3f(x) \lb h,h,h \rb| \leq \kappa ( \nabla^2 f(x) \lb h,h \rb)^{3/2}.$$
	If $\kappa = 2$, $f$ is \emph{standard self-concordant}. Any self-concordant function can be scaled to be standard self-concordant; the scaled function $\frac{1}{4}\kappa^2 f$ is standard self-concordant. Hence, we assume all self-concordant functions have $\kappa = 2$, unless stated otherwise.
\end{DEF}
There is also an equivalent definition which is frequently useful.
\begin{THM}[Lemma 4.1.2, \cite{N_ICP}]\label{equiv_scf}
	A closed convex function $f$ is self-concordant if and only if for every $x \in \RR^n$ and all $u_1, u_2, u_3 \in \RR^n$, we have
	$$| \nabla^3f(x) \lb u_1, u_2, u_3 \rb| \leq 2 \prod_{i=1}^3  \|u_i\|_x.$$
\end{THM}

The next inequalities are fundamental for self-concordant functions. These results are well known (see \cite[\S 4.1.4]{N_ICP}), but for completeness, we provide a proof.
\begin{LMA}\label{scineq}
	Let $f$ be standard self-concordant and strictly convex, and let $x \in \RR^n$ and $0 \neq d \in \RR^n$. Let $\delta = \|d\|_x$. Then for all $t \geq 0$,
	\begin{equation}\label{downf} f(x+td) \geq f(x) + tg(x)^Td + \delta t - \log (1 + \delta t) \end{equation}
	and
	\begin{equation}\label{downg} g(x+td)^Td \geq g(x)^Td + \frac{\delta^2t}{1 + \delta t}.\end{equation}
	For all $0 \leq t < \frac{1}{\delta}$,
	\begin{equation}\label{upf}  f(x+td) \leq f(x) + tg(x)^Td - \delta t - \log (1 - \delta t) \end{equation}
	and
	\begin{equation}\label{upg} g(x+td)^Td \leq g(x)^Td + \frac{\delta^2t}{1 - \delta t}.\end{equation}
\end{LMA}
\begin{proof}
Define $\phi: \RR \rightarrow \RR$ by $\phi(t) = d^T\nabla^2f(x+td)d$. Since $f$ has continuous third derivatives, $\phi(t)$ is continuously differentiable and from the definition of self-concordancy, its derivative satisfies
\begin{equation}\label{eq:dphi}|\phi'(t)| = |\nabla^3f(x+td) \lb d,d,d \rb| \leq 2(\nabla^2 f(x+td) \lb d,d \rb)^{3/2} = 2 \phi(t)^{3/2}.\end{equation}
Moreover, since $f$ is strictly convex and $d \neq 0$, $\phi(t) > 0$ for all $t$. Therefore, from (\ref{eq:dphi}),
\begin{equation*}| \frac{d}{dt} \phi(t)^{-1/2}| = \half |\phi(t)^{-3/2} \phi'(t)| \leq 1.\end{equation*}
Defining $\psi(s) = \restr{\frac{d}{dt} \phi(t)^{-1/2}}{t = s}$, the above inequality is equivalent to $|\psi(s)| \leq 1$. By Taylor's Theorem, there exists a point $u \in (0,t)$ such that $\phi(t)^{-1/2} - \phi(0)^{-1/2} = t\psi(u)$. Since $|\psi(u)| \leq 1$, we deduce that
$$\phi(0)^{-1/2} - t \leq \phi(t)^{-1/2} \leq \phi(0)^{-1/2} + t.$$
Note that $\delta = \phi(0)^{1/2}$. Rearranging the upper bound, we find that for all $t \geq 0$,
\begin{equation}\label{downH} \phi(t) \geq \frac{\delta^2}{(1 + \delta t)^2} .\end{equation}
Similarly, we find that for $0 \leq t < \frac{1}{\delta}$,
\begin{equation}\label{upH} \phi(t) \leq \frac{\delta^2}{(1 - \delta t)^2} .\end{equation}
Integrating (\ref{downH}) yields the inequalities (\ref{downf}), (\ref{downg}), and integrating (\ref{upH}) produces (\ref{upf}), (\ref{upg}).
\end{proof}

\section{Curvature-Adaptive Step Sizes}\label{adapted}
We define a general framework for an iterative method with step sizes determined by the local curvature. At each step, we compute a descent direction $d_k = -H_kg_k$, where $H_k$ is a positive definite matrix, and a step size
$$t_k = \frac{\rho_k}{(\rho_k + \delta_k)\delta_k},$$
where
$$\delta_k = \|d_k\|_{x_k}$$ and
$$\rho_k = g_k^TH_kg_k.$$
We then advance to the point $x_{k+1} = x_k + t_kd_k$.

We will refer to the above step size $t_k$ as the \emph{curvature-adaptive} step size, or simply the \emph{adaptive} step size. A method within our framework will be referred to as an \emph{adaptive} method. A generic method in this framework is specified in \Cref{alg:generic}.

\begin{algorithm}
	\caption{Adaptive Iterative Method}
	\label{alg:generic}
	\begin{algorithmic}[1]
		\Statex \textbf{input}: $x_0, H_0$, variant
		\For{$k = 0,1,2,\ldots$}
		\State{Set $d_k \leftarrow -H_kg_k$}
		\State{Set $\rho_k \leftarrow -g_k^Td_k$}
		\State{Set $\delta_k^2 \leftarrow d_k^TG_kd_k$}
		\State{Set $t_k \leftarrow \frac{\rho_k}{(\rho_k+\delta_k)\delta_k}$}
		\State{Set $x_{k+1} \leftarrow x_k + t_kd_k$}
		\If{variant (i): gradient descent}
		\State{$H_{k+1} \leftarrow I$}
		\EndIf
		\If{variant (ii): Newton}
		\State{$H_{k+1} = G_{k+1}^{-1}$}
		\EndIf
		\If{variant (iii): BFGS}
		\State{Use standard BFGS formula (\ref{BFGS_H}) to compute $H_{k+1}$}
		\EndIf
		\If{variant (iv): L-BFGS}
		\State{Update L-BFGS curvature pairs}
		\EndIf
		\EndFor
	\end{algorithmic}
\end{algorithm}

Note that this framework encompasses several classical methods. When $H_k = I$ for all $k$, the resulting method is gradient descent. When $H_k = G_k^{-1}$, we recover the \emph{damped Newton method} proposed by Nesterov. When $H_k$ is an approximation of $G_k^{-1}$ obtained by applying a quasi-Newton updating formula, the resulting method is a quasi-Newton method. In particular, we will focus on the case where $H_k$ evolves according to the BFGS update formula. We also note that in all variants other than the damped Newton method, we do not access the full Hessian matrix $G_k$ at any step, but only the action of  $G_k$ on the direction $d_k$, which typically requires a computational effort similar to that required to compute the gradient $g_k$.

Using the results of \Cref{scf_basic}, we now show that the curvature-adaptive step size $t_k = \frac{\rho_k}{(\rho_k + \delta_k)\delta_k}$ in \Cref{alg:generic} maximizes a lower bound on the decrease in $f$ obtained by taking a step in the direction $d_k$.
\begin{LMA}\label{main_ineq}
Suppose $f$ is self-concordant and strictly convex. At iteration $k$ of \Cref{alg:generic}, taking the step $t_kd_k$, where $d_k = -H_kg_k$ and $t_k = \frac{\rho_k}{(\rho_k + \delta_k)\delta_k}$, yields the point $x_{k+1} = x_k + t_kd_k$ at which the objective function $f(x_{k+1})$ satisfies
\begin{equation}\label{eq:fdec} f(x_{k+1}) \leq f(x_k) - \omega(\eta_k)\end{equation}
where
$$\eta_k = \frac{\rho_k}{\delta_k}$$
and $\omega:\RR \rightarrow \RR$ is the function $\omega(z) = z - \log(1+z)$.

Moreover, the step size $t_k$ minimizes the upper bound (\ref{upf}) on $f(x_{k+1})$ provided by \Cref{scineq}.
\end{LMA}
\begin{proof}
We fix the index $k$ and omit the subscripts for brevity. First, observe that
$$0 \leq t = \frac{\rho}{(\rho + \delta)\delta} < \frac{1}{\delta}.$$
Therefore, we can apply inequality (\ref{upf}) to $f(x+td)$. Noting that $\rho = -g^Td$, (\ref{upf}) can be written as $f(x+td) \leq f(x) - \Delta(t)$
where $\Delta(\cdot)$ is defined to be the function $\Delta(\tau) = (\rho + \delta)\tau + \log(1 - \delta \tau)$. For the curvature-adaptive step size $t$, it is easily verified that
$$\Delta(t) = \Delta \left( \frac{\rho}{(\rho+\delta)\delta} \right) = \frac{\rho}{\delta} + \log\left(\frac{\delta}{\rho+\delta}\right) = \frac{\rho}{\delta} - \log\left(1 + \frac{\rho}{\delta}\right) =  \omega(\eta).$$

Furthermore, for $0 \leq \tau < \frac{1}{\delta}$, $\frac{d}{d\tau} \Delta(\tau) = \rho + \delta - \frac{\delta}{1 - \delta \tau}$ and $\frac{d^2}{d\tau^2} \Delta(\tau) = -\frac{\delta^2}{(1 - \delta \tau)^2}$. We find that $\frac{d}{d\tau} \Delta(t) = 0$ and $\frac{d^2}{d\tau^2}\Delta(t) \leq 0$, which implies that $\Delta$ is maximized at $\tau = t = \frac{\rho}{(\rho + \delta)\delta}$.
\end{proof}

Since $\omega(\eta) = \eta - \log(1 + \eta)$ is positive for all $\eta > 0$, it follows that if $\limsup_k \eta_k > 0$, then $f(x_k) \rightarrow  -\infty$. This simple fact will be crucial in our convergence analysis. 

\begin{LMA}\label{etazero}
If, in addition to the assumptions in \Cref{main_ineq}, $f$ is bounded below, then $\eta_k = \frac{\rho_k}{\delta_k} \rightarrow 0$ for any of the adaptive variants in \Cref{alg:generic}.
\end{LMA}
\begin{proof}
By \Cref{main_ineq}, $f(x_k)$ satisfies $f(x_k) \leq f(x_0) - \sum_{j=0}^{k-1} \omega(\eta_j)$. Suppose that $\limsup_k \eta_k > 0$. Since the function $\omega(\eta)$ is positive and monotonically increasing for $\eta > 0$, we have $\limsup_k \omega(\eta_k) = \omega(\limsup_k \eta_k) > 0$. Hence $f(x_k) \rightarrow -\infty$, a contradiction.
\end{proof}

In terms of $g_k, H_k$, and $G_k$, the adaptive step size $t_k$ can be expressed as
$$t_k = \frac{g_k^TH_kg_k}{g_k^TH_kG_kH_kg_k + g_k^TH_kg_k\sqrt{g_k^TH_kG_kH_kg_k}}.$$
This formula relates $t_k$ to the local curvature. When the curvature of $f$ in the direction $d_k = -H_kg_k$ is relatively flat, the local norm $\|d_k\|_{x_k} = \sqrt{g_k^TH_kG_kH_kg_k}$ is small, and the adaptive step size $t_k$ will be large. Conversely, when the curvature of $f$ in the direction $d_k$ is steep, $t_k$ will be small. Intuitively, this is precisely the desired behavior for a step size, since we wish to take larger steps when the function changes slowly, and smaller steps when the function changes rapidly.

\section{Scaled Gradient Methods}\label{boundedH}
We first consider the class of methods where the matrices $H_k$ are positive definite and uniformly bounded above and below. That is, there exist positive constants $\lambda, \Lambda$ such that for every $k \geq 0$,
\begin{equation}\label{hbound} \lambda I \preceq H_k \preceq \Lambda I. \end{equation}
The convergence analysis is rather straightforward, as seen in the proofs of the following two theorems for these methods.

\begin{THM}\label{Gzero}
If $f$ is self-concordant, strictly convex, bounded below, and the Hessian satisfies $G(x) \preceq MI$ on the level set $\Omega = \{x: f(x) \leq f(x_0)\}$, then any adaptive method (\Cref{alg:generic}) for which the matrices $H_k$ satisfy \cref{hbound} converges globally in the sense that $\lim\limits_{k \rightarrow \infty} \|g_k\| = 0$.
\end{THM}
\begin{proof}
\textcolor{black}{Since $H_k$ is positive definite, $H_k^{1/2}$ exists and we may define $z_k = H_k^{1/2}g_k$}. Observe that
\begin{equation}\label{etag}
\eta_k = \frac{ g_k^TH_kg_k}{\sqrt{ g_k^TH_kG_kH_kg_k}} = \frac{z_k^Tz_k}{\sqrt{z_k^T(H_k^{1/2}G_kH_k^{1/2})z_k}} \geq \frac{\|z_k\|}{\sqrt{\Lambda M}} \geq \sqrt{\frac{\lambda}{\Lambda M}} \|g_k\|
\end{equation}
\textcolor{black}{where we have used the fact that the maximum eigenvalue of $H_k^{1/2}G_kH_k^{1/2}$ is bounded by $\Lambda M$.} By \Cref{etazero}, $\eta_k \rightarrow 0$. Therefore $\|g_k\| \rightarrow 0$.
\end{proof}

If in addition, $f$ is strongly convex with $mI \preceq G(x)$ for $m > 0$, then an adaptive method satisfying \cref{hbound} is globally $R$-linearly convergent. The proof uses the fact that strongly convex functions satisfy the \emph{Polyak-\L{}ojasiewicz inequality}, which is stated in the following well known lemma (e.g., see \cite{POW, BLOCKBFGS}).
\begin{LMA}\label{gTof}
If $f$ is strongly convex with $mI \preceq G(x)$, and $x_\ast$ is the unique minimizer of $f$, then $\|g(x)\|^2 \geq 2m(f(x) - f(x_\ast))$.
\end{LMA}

We are now ready to prove the $R$-linear convergence of adaptive scaled gradient methods.

\begin{THM}\label{RLinear}
If $f$ is self-concordant and strongly convex (so there exist constants $0 < m \leq M$ such that $mI \preceq G(x) \preceq MI$ for all $x \in \Omega$), then an adaptive method (\Cref{alg:generic}) for which the matrices $H_k$ satisfy \cref{hbound} is globally $R$-linearly convergent. That is, there exists a positive constant $\gamma < 1$ such that $f(x_{k+1}) - f(x_\ast) \leq \gamma (f(x_k) - f(x_\ast))$ for all $k$.
\end{THM}
\begin{proof}
Since $\eta_k \rightarrow 0$ by \Cref{etazero}, the sequence $\{\eta_k\}_{k=0}^\infty$ is bounded. Let $\Gamma = \sup_k \eta_k < \infty$, and let $c = \frac{1}{2(1+\Gamma)}$. Observe that $\omega(z) = z - \log(1+z) \geq cz^2$ for $0 \leq z \leq \Gamma$, as $\omega(0) = 0$ and $\frac{d}{dz} (\omega(z) - cz^2) = \frac{z(1 - 2c - 2cz)}{1 + z}$, which is non-negative for $0 \leq z \leq \Gamma$. Hence, since $\eta_k \leq \Gamma$ for all $k$, we have
\begin{align*}
f(x_{k+1}) - f(x_\ast) \leq f(x_k) - f(x_\ast) - \omega(\eta_k) &\leq f(x_k) - f(x_\ast) - c\eta_k^2 \\
&\leq f(x_k) - f(x_\ast) -  \frac{c\lambda}{\Lambda M} \|g(x_k)\|^2 \\
&\leq \left( 1 - \frac{\lambda m}{\Lambda(1+\Gamma)M}\right) (f(x_k) - f(x_\ast))
\end{align*}
where the first line follows from inequality (\ref{upf}), the second from inequality (\ref{etag}), and the third from \Cref{gTof}. Taking $\gamma = 1 - \frac{\lambda m}{\Lambda(1+\Gamma)M}$, we obtain the desired $R$-linear convergence.
\end{proof}

\subsection{Adaptive Gradient Descent}\label{gd}
When $H_k = I$ for all $k$ in \Cref{alg:generic}, the method corresponds to gradient descent with adaptive step sizes that incorporate second-order information. This strategy for selecting analytically computable step sizes may have several advantages in practice. Using second-order information allows a better local model of the objective function. The classical analysis of gradient descent with a fixed step size also generally requires a sufficiently small step size in order to guarantee convergence. This step size is a function of the Lipschitz constant for the gradient $g(x)$, which is either unknown or impractical to compute. The step size needed to ensure convergence in theory is also often impractically tiny, leading to slow convergence in practice. For the class of self-concordant functions, an adaptive step size can be easily computed without knowledge of any constants, and still provides a theoretical guarantee of convergence, which is a significant advantage.

A proximal gradient descent method with adaptive step sizes was studied by Tran-Dinh et al. \cite{TRANDINH}, who proved the method to be globally convergent for self-concordant functions, and locally $R$-linearly convergent for strongly convex self-concordant functions. However, our convergence analysis above employs different techniques from those in \cite{TRANDINH}, and in particular, we obtain the following theorem, which shows that the adaptive gradient descent method is globally $R$-linearly convergent, as an immediate corollary of \Cref{Gzero} and \Cref{RLinear}:
\begin{THM}
Suppose that $f$ is self-concordant, strictly convex, bounded below, and $G(x) \preceq MI$ on the level set $\Omega = \{x \in \RR^n: f(x) \leq f(x_0)\}$. Then the adaptive gradient descent method converges in the sense that $\lim_{k \rightarrow \infty} \|g_k\| = 0$. Furthermore, if $f$ is strongly convex on $\Omega$, then the adaptive gradient descent method is globally $R$-linearly convergent.
\end{THM}

\subsection{Adaptive L-BFGS}
The limited-memory BFGS algorithm (L-BFGS, \cite{LBFGS}) stores a fixed number of previous \emph{curvature pairs} $(s_k, y_k)$, where $s_k = x_{k+1} - x_k$ \textcolor{black}{and} $y_k = g_{k+1} - g_k$, and computes $d_k = -H_kg_k$ from the curvature pairs using a two-loop recursion \cite{TWOLOOP}. It is well known that L-BFGS satisfies \cref{hbound}. In \cite{GGR}, the following bounds are obtained.

\begin{THM}[Lemma 1, \cite{GGR}]
Suppose that $f$ is strongly convex, with $mI \leq G(x) \leq MI$. Let $\ell$ be the number of curvature pairs stored by the L-BFGS method. Then the matrices $H_k$ satisfy
$$\lambda I \preceq H_k \preceq \Lambda I,$$
where $\lambda = (1 + \ell M)^{-1}$ and $\Lambda = (1 + \sqrt{\kappa})^{2\ell}\left( 1 + \frac{1}{m(2\sqrt{\kappa} + \kappa)} \right)$ for $\kappa = M/m$.
\end{THM}
Hence, it follows immediately from \Cref{Gzero} and \Cref{RLinear} that:
\begin{THM}
Suppose that $f$ is self-concordant, strongly convex, and $mI \preceq G(x) \preceq MI$ on the level set $\Omega = \{x \in \RR^n: f(x) \leq f(x_0)\}$. Then the adaptive L-BFGS method is globally $R$-linearly convergent.
\end{THM}

We note that, as with gradient descent, it is well known that the L-BFGS method converges if inexact Armijo-Wolfe line searches are performed, or a sufficiently small fixed step size, that depends on the Lipschitz constant of $g(x)$, is used.

\section{Adaptive BFGS}\label{BFGS}
If $H_k$ is chosen to approximate $(\nabla^2 f(x_k))^{-1}$, then we obtain quasi-Newton methods with adaptive step sizes. In particular, we may iteratively update $H_k$ using the BFGS update formula, which we briefly describe. Let $s_k = x_{k+1} - x_k$ and $y_k = g_{k+1} - g_k$. The BFGS update sets $H_{k+1}$ to be the nearest matrix to $H_k$ (in a variable metric) satisfying the \emph{secant equation} $H_{k+1}y_k = s_k$\cite{G1970}. It is well known that $H_{k+1}$ has the following expression in terms of $H_k, s_k$ and $y_k$:
\begin{equation}\label{BFGS_H} H_{k+1} = \frac{s_ks_k^T}{y_k^Ts_k} + \left(I - \frac{s_ky_k^T}{y_k^Ts_k}\right)H_k\left(I - \frac{y_ks_k^T}{y_k^Ts_k}\right).\end{equation}

\subsection{Superlinear Convergence of Adaptive BFGS}

The convergence analysis of the classical BFGS method \cite{POW, BNY} assumes that the method uses inexact line searches satisfying the \emph{Armijo-Wolfe} conditions: for constants $c_1, c_2 \in (0,1)$ with $c_1 < \half$ and $c_1 < c_2$, the step size $t_k$ should satisfy
\begin{equation}\label{armijo} f(x_k + t_kd_k) \leq f(x_k) + c_1 t_k g_k^Td_k \hspace{2em} \text{(Armijo condition)}\end{equation}
and
\begin{equation}\label{wolfe} g(x_k + t_kd_k)^Td_k \geq c_2 g_k^Td_k. \hspace{2em} \text{(Wolfe condition)} \end{equation}

Under the assumption of Armijo-Wolfe line searches, Powell \cite{POW} proves the following global convergence theorem for BFGS.
\begin{THM}[Lemma 1, \cite{POW}]\label{POW_C}
If the BFGS algorithm with Armijo-Wolfe inexact line search is applied to a convex function $f(x)$ that is bounded below, if $x_0$ is any starting vector and $H_0$ is any positive definite matrix, and if the Hessian $G(x)$ satisfies $G(x) \preceq MI$ for all $x$ in the level set $\Omega = \{x: f(x) \leq f(x_0)\}$, then the limit
\begin{equation}\label{liminfg} \liminf_{k \rightarrow \infty} \|g_k\| = 0\end{equation}
is obtained.
\end{THM}

In our setting, $f$ is a self-concordant and strictly convex function that is bounded below and satisfies $G(x) \preceq MI$. In order to prove that adaptive BFGS is convergent in the sense of the limit (\ref{liminfg}), it suffices to show that the adaptive step sizes $t_k$ satisfy the Armijo condition for any $c_1 < \frac{1}{2}$, and eventually satisfy the Wolfe condition for any $c_2 < 1$ (i.e. there exists some $k_0$ such that the Wolfe condition is satisfied for all $k \geq k_0$). Specifically, we prove the following two theorems that apply to \emph{every} adaptive method described by \Cref{alg:generic}.

\begin{THM}
Let $f$ be self-concordant and strictly convex. The curvature-adaptive step size $t_k = \frac{\rho_k}{(\rho_k + \delta_k)\delta_k}$ satisfies the Armijo condition for any $c_1 \leq \half$.
\end{THM}
\begin{proof}
Let $c_1 \leq \half$. We aim to prove that $f(x_{k+1}) \leq f(x_k) + c_1 t_k g_k^Td_k$. By \Cref{main_ineq}, $f(x_{k+1}) \leq f(x_k) - \omega(\eta_k)$. Therefore, it suffices to prove that
$$\omega(\eta_k) \geq -\half t_kg_k^Td_k.$$
For brevity, we omit the index $k$. Notice that
$$-tg^Td = tg^THg = t\rho = \frac{\rho^2}{(\rho + \delta)\delta} = \frac{\rho^2/\delta^2}{\rho/\delta+1} = \frac{\eta^2}{1 + \eta}.$$
Therefore, we must prove that for $\eta \geq 0$,
$$\omega(\eta) = \eta - \log(1 + \eta) \geq \half \frac{\eta^2}{1 + \eta}.$$
Define $\zeta(z) = z - \log(1 + z) - \half \frac{z^2}{1 + z}$. Observe that $\zeta(0) = 0$ and
$$\frac{d}{dz} \zeta(z) = 1 - \frac{1}{1 + z} - \half \frac{z^2 + 2z}{(1 + z)^2} = \half \frac{z^2}{(1+z)^2}.$$
Since $\frac{d}{dz} \zeta(z) \geq 0$ for all $z \geq 0$, we conclude that $\omega(\eta) \geq \half \frac{\eta^2}{1 + \eta}$ for all $\eta \geq 0$. This completes the proof.
\end{proof}

\begin{THM}
Let $f$ be self-concordant, strictly convex, and bounded below. Suppose that $\{x_k\}_{k=0}^\infty$ is a sequence of iterates generated by \Cref{alg:generic}. For any $0 < c_2 < 1$, there exists an index $k_0$ such that for all $k \geq k_0$, the curvature-adaptive step size $t_k$ satisfies the Wolfe condition.
\end{THM}
\begin{proof}
We aim to prove that $g_{k+1}^Td_k \geq c_2 g_k^Td_k$. This is equivalent to $g(x_k + t_kd_k)^Td_k - g(x_k)^Td_k \geq -(1-c_2)g(x_k)^Td_k = (1-c_2)\rho_k$. By inequality (\ref{downg}) with $\delta = \delta_k$ and $t = t_k$, we have
\begin{equation}\label{eq:wolfedg}
g(x_k + t_kd_k)^Td_k - g(x_k)^Td_k \geq \frac{\delta_k^2t_k}{1 + \delta_k t_k} = \frac{\delta_k\rho_k}{2\rho_k + \delta_k} = \frac{1}{1 + 2\eta_k}\rho_k.\end{equation}
Since $f$ is bounded below, \Cref{etazero} implies that $\eta \rightarrow 0$, and hence there exists some $k_0$ such that $\frac{1}{1 + 2\eta_k} \geq 1 - c_2$ for all $k \geq k_0$.
\end{proof}

Note that the assumption of strict convexity also implies that $y_k^Ts_k > 0$, so the BFGS update is well-defined.

We can now immediately apply \Cref{POW_C} to deduce that adaptive BFGS is convergent. Since there always exists an index $k_0$ such that the Armijo-Wolfe conditions are satisfied for all $k \geq k_0$, we can consider the subsequent iterates $\{x_k\}_{k=k_0}^\infty$ as arising from the classical BFGS method with initial matrix $H_{k_0}$.
\begin{THM}\label{global}
Let $f$ be self-concordant, strictly convex, bounded below, \textcolor{black}{whose} Hessian satisfies $G(x) \preceq MI$ for all $x \in \Omega$. Then for the adaptive BFGS method, $\liminf_{k \rightarrow \infty} \|g_k\| = 0$.
\end{THM}

It is also possible to directly prove \Cref{global} by analyzing the evolution of the trace and determinant of $H_k$, but the resulting proof, which is quite long, does not provide clarity on the essential properties of the adaptive step size.

It is well known that if the objective function $f$ is strongly convex, then the classical BFGS method converges $Q$-superlinearly. Let us now assume that $f$ is strongly convex, so there exist constants $0 < m \leq M$ with $mI \preceq G(x) \preceq MI$ for all $x \in \Omega$. Let $x_\ast$ denote the unique minimizer of $f$.

\begin{THM}[Lemma 4, \cite{POW}]
	Let $f$ be strongly convex, and let $\{x_k\}_{k=0}^\infty$ be the sequence of iterates generated by the BFGS method with inexact Armijo-Wolfe line searches. Then $\sum_{k=0}^\infty \|x_k - x_\ast\| < \infty$.
\end{THM}

Since the adaptive step size $t_k$ eventually satisfies the Armijo-Wolfe conditions, the same holds for BFGS with adaptive step sizes.
\begin{THM}\label{summable}
	Let $f$ be self-concordant and strongly convex. The sequence of iterates $\{x_k\}_{k=0}^\infty$ produced by adaptive BFGS satisfies $\sum_{k=0}^\infty \|x_k - x_\ast\| < \infty$.
\end{THM}

In the proof of superlinear convergence for the classical BFGS method, it is assumed that the Hessian $G(x)$ is Lipschitz continuous. However, it is unnecessary to make this assumption in our setting, as $G(x)$ is necessarily Lipschitz when $f$ is self-concordant and $G(x)$ is bounded above. This fact is not difficult to establish, but we provide a proof for completeness.

\begin{THM}\label{scf_cts}
	If $f$ is self-concordant and satisfies $G(x) \preceq MI$ for all $x \in \Omega$, then $G(x)$ is Lipschitz continuous on $\Omega$, with constant $2M^{3/2}$.
\end{THM}
\begin{proof}
	Let $x,y \in \Omega$, and let $e = x - y$. Let $v \in \RR^n$ be any unit vector. By Taylor's Theorem, we have
	$$v^TG(x)v = v^TG(y)v + \int_0^1 \nabla^3 f(y+\tau e) \lb v, v, e \rb d\tau. $$
	Hence, by \Cref{equiv_scf},
	\begin{align*}
	|v^T(G(x) - G(y))v| &\leq \int_0^1 |\nabla^3 f(y+\tau e)\lb v,v,e \rb | d\tau \\
	&\leq 2\int_0^1 v^T G(y+\tau e) v \sqrt{ e^TG(y+\tau e) e} d\tau \\
	&\leq 2\int_0^1 M^{3/2} \|e\| d\tau = 2M^{3/2} \|x - y\|.
	\end{align*}
	Therefore, the eigenvalues of $G(x) - G(y)$ are bounded in norm by $2M^{3/2}\|x - y\|$. It follows that $\|G(x) - G(y)\| \leq 2M^{3/2}\|x - y\|$, so $G(x)$ is Lipschitz continuous.
\end{proof}

It is well known that the BFGS method is invariant under an affine change of coordinates, so we may assume without loss of generality that $G(x_\ast) = I$. This corresponds to considering the function $\wt{f}(\wt{x}) = f(G(x_\ast)^{-1/2}\wt{x})$ and performing a change of coordinates $\wt{x} = G(x_\ast)^{1/2}x$. By \cite[Theorem 4.1.2]{N_ICP}, the function $\wt{f}$ is also self-concordant, with the same $\kappa$ as for $f$.

To complete the proof of superlinear convergence, we use results established by Dennis and Mor\'{e} \cite{DM1} and Griewank and Toint \cite{GT}. In \cite[\S4]{GT}, Griewank and Toint prove that, given \Cref{summable} and Lipschitz continuity of $G(x)$ (\Cref{scf_cts}), the following limit holds:
\begin{equation}\label{dm} \lim_{k \rightarrow \infty} \frac{ \|(B_k - I)d_k\|}{\|d_k\|} = 0 \end{equation}
Furthermore, the argument in \cite{GT} shows that both $\{ \|B_k\| \}_{k=0}^\infty$ and $\{ \|H_k\| \}_{k=0}^\infty$ are bounded. Writing $B_kd_k = -g_k$ and $-d_k = H_kg_k$, \textcolor{black}{ and using the fact that $\|d_k\| \leq \|H_k\|\|g_k\| \leq \Gamma \|g_k\|$, where $\Gamma$ is an upper bound on the sequence of norms $\{\|H_k\|\}_{k=0}^\infty$}, we have an equivalent limit
\begin{equation}\label{dm2} \lim_{k \rightarrow \infty} \frac{ \|H_kg_k - g_k\|}{\|g_k\|} = 0 \end{equation}
This enables us to prove that the adaptive step sizes $t_k$ converge to 1, which is necessary for superlinear convergence.
\begin{THM}\label{t1}
The curvature-adaptive step sizes $t_k$ in the adaptive BFGS method converge to 1.
\end{THM}
\begin{proof}
We omit the index $k$ for brevity, and define $u = Hg - g$. Since $t$ can be expressed as $t = \frac{\eta/\delta}{1 + \eta}$, and we have from \Cref{etazero} that $\eta \rightarrow 0$, it suffices to show that $\frac{\eta}{\delta}$ converges to 1.
\begin{align*}
\frac{\eta}{\delta} = \frac{\rho}{\delta^2} &= \frac{ g^THg}{g^THGHg} \\
&= \frac{ g^Tg + g^Tu}{g^TGg + 2g^TGu + u^TGu} \\
&= \frac{1 + \frac{g^Tu}{g^Tg}}{\frac{g^TGg}{g^Tg} + 2\frac{g^TGu}{g^Tg} + \frac{u^TGu}{g^Tg}}
\end{align*}
The limit (\ref{dm2}) implies that $\frac{\|u\|}{\|g\|} \rightarrow 0$. Hence, the Cauchy-Schwarz inequality and the upper bound $G(x) \preceq MI$ imply that $\frac{g^Tu}{g^Tg}, \frac{g^TGu}{g^Tg}, \frac{u^TGu}{g^Tg}$ converge to 0. Since $G = G(x_k)$ and $x_k \rightarrow x_\ast$, we have $G \rightarrow I$, and therefore $\frac{g^TGg}{g^Tg} \rightarrow 1$. It follows that $\frac{\eta}{\delta}$, and therefore $t$, converges to 1.
\end{proof}

We now make a slight modification to the Dennis-Mor\'{e} characterization of superlinear convergence. \textcolor{black}{Using the triangle inequality twice and the fact that $G(x_\ast) = I$, we obtain}
\begin{align*}
\frac{\|(B_k - I)s_k)\|}{\|s_k\|} &= \frac{ \| t_kg_{k+1} - t_kg_k - G(x_\ast)s_k \textcolor{red}{-} t_kg_{k+1}\|}{\|s_k\|} \\
&\geq t_k\frac{ \|g_{k+1}\|}{\|s_k\|} - \frac{ \|t_kg_{k+1} - t_kg_k - t_kG(x_\ast)s_k \textcolor{red}{-} (1-t_k)G(x_\ast)s_k\|}{\|s_k\|} \\
&\geq t_k \frac{ \|g_{k+1}\|}{\|s_k\|} - \frac{ t_k \|\int_0^1 (G(x_k + \tau s_k) - G(x_\ast))s_k d\tau \|}{\|s_k\|} - |1-t_k| \frac{ \|G(x_\ast)s_k\|}{\|s_k\|}.
\end{align*}
Rearranging, and applying \Cref{scf_cts},
\begin{equation}\label{finalgs}
\frac{ \|g_{k+1}\|}{\|s_k\|} \leq \frac{1}{t_k} \frac{\|(B_k - I)s_k)\|}{\|s_k\|} + 2M^{3/2}\max\{ \|x_k - x_\ast\|, \|x_{k+1} - x_\ast\|\} + \frac{|1 - t_k|}{t_k}M .
\end{equation}
Since $x_k \rightarrow x_\ast$ by \Cref{global} and $t_k \rightarrow 1$ by \Cref{t1}, both of the latter terms converge to 0. Finally, \cref{dm} implies that $\frac{\|(B_k - I)s_k\|}{\|s_k\|}$ converges to 0, so it follows from \cref{finalgs} that $\frac{ \|g_{k+1}\|}{\|s_k\|}$ converges to 0.

Since $f$ is strongly convex, $\|g(x)\| \geq m \|x - x_\ast\|$. Hence, we find that
\begin{align*}
\frac{ \|g_{k+1}\|}{\|s_k\|} \geq \frac{m \|x_{k+1} - x_\ast\|}{\|x_{k+1} - x_\ast\| + \|x_k - x_\ast\|},
\end{align*}
which implies that $\frac{ \|x_{k+1} - x_\ast\|}{\|x_k - x_\ast\|} \rightarrow 0$. Thus, we have the following:
\begin{THM}
Suppose that $f$ is self-concordant, and strongly convex on $\Omega = \{x \in \RR^n: f(x) \leq f(x_0)\}$. Then the adaptive BFGS method converges $Q$-superlinearly.
\end{THM}

By the same reasoning, the results in \cite{BNY} and \cite{GT} imply that these convergence theorems also hold for the adaptive versions of the quasi-Newton methods in \emph{Broyden's convex class}, with the exception of the DFP method. The adaptive versions of the Block BFGS methods proposed in \cite{BLOCKBFGS} can also be shown to be $Q$-superlinearly convergent.

\subsection{Hybrid Step Selection}\label{sub:hybrid}
\textcolor{black}{Consider the damped Newton method of Nesterov, which is obtained by setting $H_k = G_k^{-1}$. This yields $\rho_k = g_k^TG_k^{-1}g_k$ and $\delta_k = \sqrt{g_kG_k^{-1}G_kG_k^{-1}g_k} = \sqrt{\rho}$, whence $\eta = \rho/\delta = \delta$. The curvature-adaptive step size $t$ then reduces to
$$t = \frac{\eta/\delta}{1 + \eta} = \frac{1}{1 + \delta}.$$
When $\delta$ is large (for example, if the initial point $x_0$ is chosen poorly), then the curvature-adaptive step size may be very small, even when the inverse Hessian approximation $H_k$ is good. This conservatism is the price of the curvature-adaptive step size guaranteeing global convergence (in contrast to Newton's method, which is \emph{not} globally convergent, and to gradient descent, which may diverge if the step size is too large).} A small step $t_kd_k$ is likely to result in $t_{k+1}$ also being small\footnotemark. Thus, when the initial $\delta$ is large, a method using adaptive step sizes may produce a long succession of small steps. This suggests the following heuristic for selecting step sizes:
\begin{enumerate}
\item Select a set $T_k$ of candidate step sizes for $t_k$.
\item At step $k$, test the elements of $T_k$ in order until a candidate step size is found which satisfies the Armijo condition (\ref{armijo}).
\item If no element of $T_k$ satisfies the Armijo condition, then set $t_k$ to be the adaptive step size.
\end{enumerate}
For instance, in our numerical experiments reported in \Cref{numerical}, we take $T_k$ to be $(1, \frac{1}{4}, \frac{1}{16})$ for all $k$. This allows the method to take steps of size $t_k = 1$ when 1 satisfies the Armijo condition, which is desirable for reducing the number of iterations needed before superlinear convergence kicks in.

\footnotetext{\textcolor{black}{
As an illustrative example, consider applying the damped Newton method to the quadratic function $f(x) = \frac{1}{2}\|x\|^2$. Since $d_k = -x_k$ and $\delta_k = \|x_k\|$, we have $t_k = \frac{1}{1 + \|x_k\|}$ and $x_{k+1} = \frac{\|x_k\|}{1 + \|x_k\|} x_k$. If $\|x_0\|$ is large, then it is clear that the damped Newton method will take many tiny steps until $\|x_k\|$ is sufficiently reduced. This is in stark contrast to Newton's method, which reaches the minimizer after a single step.}}

We refer to this scheme as \emph{hybrid step selection}. For a proper choice of $T_k$, hybrid step selection avoids the disadvantage of exclusively using adaptive step sizes, where the step size may be small for many iterations. It will also generally be more efficient to compute than a full line search, since no more than $|T_k|$ candidate step sizes are tested before switching to the adaptive step size.

\section{Numerical Experiments}\label{numerical}
To compare our adaptive methods to classical algorithms, we solve several binary classification problems using \emph{logistic regression}. In these problems, the objective function to be minimized has the form
\begin{equation}
\label{logreg}
L(w) = \frac{1}{N} \sum_{i=1}^N \log(1 + \exp(-y_i x_i^Tw)) + \frac{1}{2N} \|w\|_2^2.\end{equation}
where the training data $\{(x_1,y_1),\ldots, (x_N, y_N)\}$ consists of feature vectors $x_i \in \RR^n$ and classifications $y_i \in \{-1,+1\}$. Zhang and Xiao \cite{ZX} showed that the logistic regression objective function $L(w)$ is self-concordant.
\begin{THM}[Lemma 1, \cite{ZX}]\label{lrsc}
	Let $B = \max_i \|x_i\|$. The scaled function $\frac{B^2 N}{4}L(w)$ is standard self-concordant.
\end{THM}

In our tests, we compared seven algorithms:
\begin{enumerate}
	\item BFGS with adaptive step sizes (BFGS-A).
	\item BFGS with Armijo-Wolfe line search (BFGS-LS).
	\item BFGS with hybrid step selection (BFGS-H), using $T_k = (1, \frac{1}{4}, \frac{1}{16})$.
	\item L-BFGS with adaptive step sizes (LBFGS-A), using the past $\ell = \min \{\frac{n}{2}, 20\}$ curvature pairs.
	\item L-BFGS with Armijo-Wolfe line search (LBFGS-LS), using the past $\ell = \min \{ \frac{n}{2}, 20\}$ curvature pairs.
	\item Gradient descent with adaptive step sizes (GD-A).
	\item Gradient descent with Armijo-Wolfe line search (GD-LS).
\end{enumerate}

An initial Hessian approximation $H_0$ must be provided for the BFGS and L-BFGS methods. It is easy, but not necessarily effective, to simply take $H_0 = I$. Another common strategy for initializing $H_0$, described in \cite{NW}, that is often quite effective, is to take $H_0 = I$ on the first step, and then, before performing the first BFGS update (\ref{BFGS_H}), scale $H_0$:
\begin{equation}\label{Idscale}H_0 \leftarrow \frac{y_0^Ts_0}{y_0^Ty_0} I .\end{equation}
It is easy to verify that the scaling factor $y_0^Ts_0/y_0^Ty_0$ lies between the smallest and largest eigenvalues of the inverse of the average Hessian $\ov{G} = \int_0^1 G(x_0 + \tau s_0) d\tau$ along the initial step. 

Similarly, for the L-BFGS method, the initial matrix used at step $k+1$ in the two-loop recursion is chosen as:
$$H_0 \leftarrow \frac{y_k^Ts_k}{y_k^Ty_k} I.$$
We refer to this as \emph{identity scaling}.

The line search used the \texttt{WolfeLineSearch} routine from the minFunc software package \cite{MINFUNC}. The Armijo-Wolfe parameters were $c_1 = 0.1, c_2 = 0.75$, and the line search was configured to use an initial step size $t = 1$ and perform quadratic interpolation (\texttt{LS\_interp = 1}, \texttt{LS\_multi = 0}).

We chose six data sets from LIBSVM \cite{LIBSVM} with a variety of dimensions, which are listed in \Cref{tab:ds}. We plot the progress of each algorithm as a function of CPU time used. The progress is measured by the \emph{log gap} $\log_{10}(f(w) - f(w_\ast))$, where  $w_\ast$ is a pre-computed optimal solution.  The starting point was always set to $w_0 = 0$. All algorithms were terminated when either the gradient reached the threshold $\|g(x)\| < 10^{-7}$, or after 480 seconds of CPU time.  \textcolor{black}{A brief summary of the results can be found in \Cref{tab:iterations}, which lists the number of iterations taken by the BFGS-type methods for convergence.}

Our algorithms were implemented in Matlab 2017a and run on an Intel i5-6200U processor. While the CPU time is clearly platform-dependent, we sought to minimize implementation differences between the algorithms to make the test results as comparable as possible.
\begin{table}
	\begin{tabular}{|l|c|c|}
		\hline
		Data set & $n$ & $N$ \\ \hline
		\texttt{covtype.libsvm.binary.scale} & 55 & 581012 \\ \hline
		\texttt{ijcnn1.tr} & 23 & 35000 \\ \hline
		\texttt{leu} & 7130 & 38 \\ \hline
		\texttt{rcv1\_train.binary} & 47237 & 20242 \\ \hline
		\texttt{real-sim} & 20959 & 72309 \\ \hline
		\texttt{w8a} & 301 & 49749 \\ \hline
	\end{tabular}
	\caption{Data sets used in \Cref{numerical}}
	\label{tab:ds}
\end{table}

\begin{table}
	\begin{tabular}{|l|l|c|c|c|c|}
		\hline
		\multirow{2}{*}{Data set} & \multirow{2}{*}{$n$} & \multirow{2}{*}{Identity Scaling} & \multicolumn{3}{|c|}{Number of iterations} \\
		\cline{4-6} & & & BFGS-A & BFGS-LS & BFGS-H \\ \hline
		\multirow{2}{*}{\texttt{covtype.libsvm.binary.scale}} & \multirow{2}{*}{55} & No & 844 & 80 & 126 \\
		\cline{3-6} &  & Yes & 1532 & 458 & 479 \\ \hline
		\multirow{2}{*}{\texttt{ijcnn1.tr}} & \multirow{2}{*}{23} & No & 286 & 36 & 66 \\
		\cline{3-6} &  & Yes & 434 & 142 & 162 \\ \hline
		\multirow{2}{*}{\texttt{w8a}} & \multirow{2}{*}{301} & No & 2254 & 240 & 637 \\
		\cline{3-6} & & Yes & 2506 & 398 & 653 \\ \hline
		\multirow{2}{*}{\texttt{leu}} & \multirow{2}{*}{7130} & No & 1197 & 95 & 293 \\
		\cline{3-6} &  & Yes & 909 & 177 & 251 \\ \hline
		\multirow{2}{*}{\texttt{rcv1\_train.binary}} & \multirow{2}{*}{47237} & No & 161 & 31 & 35 \\
		\cline{3-6} &  & Yes & 284 & 217 & 232 \\ \hline
		\multirow{2}{*}{\texttt{real-sim}} & \multirow{2}{*}{20959} & No & 356 & 44 & 55\\
		\cline{3-6} & & Yes & 592 & 247 & 317 \\ \hline
	\end{tabular}
	\caption{The number of iterations until convergence of the BFGS methods.}
	\label{tab:iterations}
\end{table}

In \Cref{fig:timeS}, we plot the results for the data sets \texttt{covtype.libsvm.binary.scale}, \texttt{ijcnn1.tr}, and \texttt{w8a}. On these problems, we implemented BFGS with a dense Hessian; that is, the matrices $H_k$ were stored explicitly and updated using the formula (\ref{BFGS_H}). In \Cref{tab:iterations}, we list the number of iterations used by the BFGS-type methods.
\begin{figure}
	\centering
	\includegraphics[clip=true, scale=0.7]{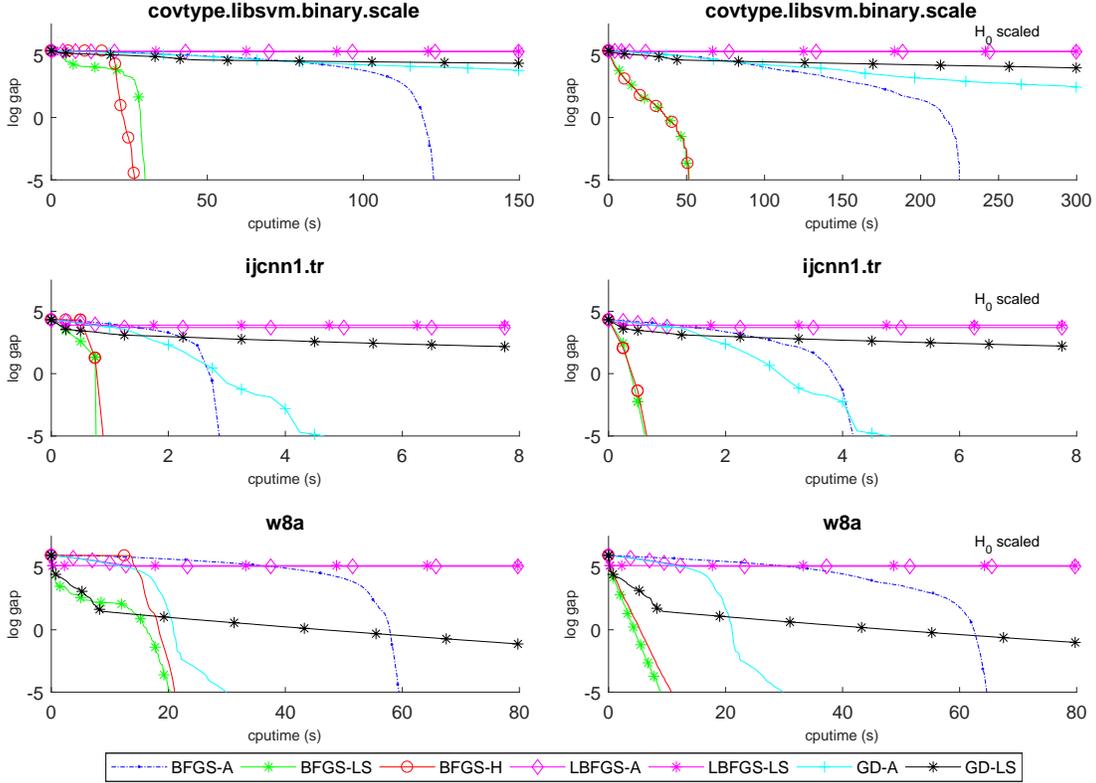}
	\caption{Experiments on problems with small $n$. The log gap is defined as $\log_{10}(f(w) - f(w_\ast))$. The loss functions are scaled to be standard self-concordant. All BFGS and L-BFGS plots on the left take $H_0 = I$, and those on the right use identity scaling.}
	\label{fig:timeS}
\end{figure}

In \Cref{fig:timeL}, we plot the results for the data sets \texttt{leu}, \texttt{rcv1\_train.binary}, and \texttt{real-sim}. These problems had a large number of variables ($n > 7000$), which made it infeasible to store $H_k$ explicitly. On these problems, BFGS was implemented using the two-loop recursion with unlimited memory, and $H_0$ was kept fixed throughout the iteration process. If the number of iterations exceeds roughly $n/4$, then this approach would in fact require more memory than storing $H_k$ explicitly. However, this never occurred in our tests, as shown in \Cref{tab:iterations}.

\begin{figure}
	\centering
	\includegraphics[clip=true, scale=0.7]{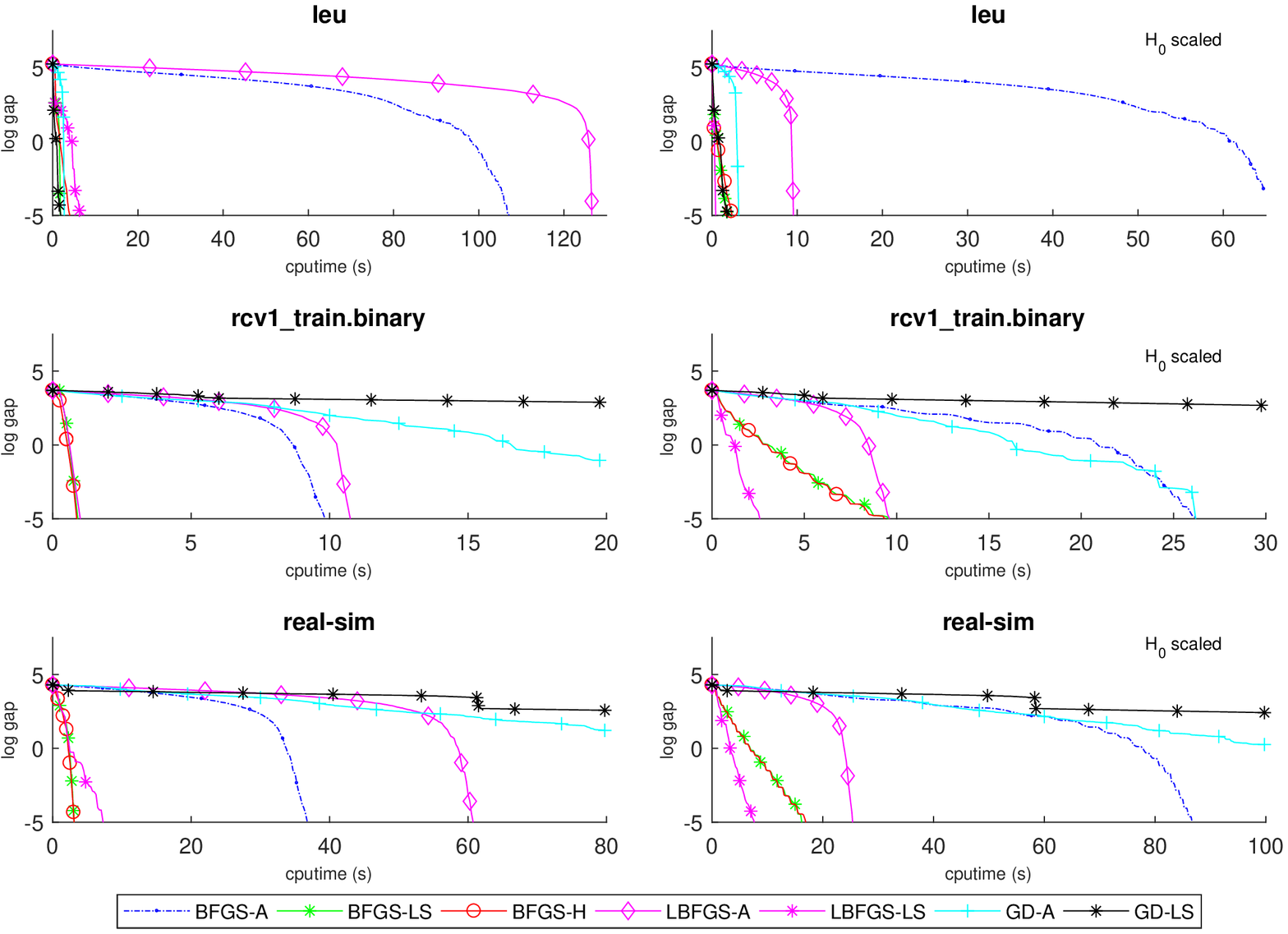}
	\caption{Experiments on problems with large $n$. The log gap is defined as $\log_{10}(f(w) - f(w_\ast))$. The loss functions are scaled to be standard self-concordant. All BFGS and L-BFGS plots on the left take $H_0 = I$, and those on the right use identity scaling.}
	\label{fig:timeL}
\end{figure}

In our tests, we found that BFGS-A required more time than BFGS-LS. Although the cost of a single step was initially lower for BFGS-A than BFGS-LS, BFGS-A often took numerous small steps in succession, making very slow progress. This situation was exactly our motivation for devising the hybrid step selection described in \Cref{sub:hybrid}, and unfortunately, appears to occur often. However, BFGS-H achieved \textcolor{black}{comparable speed to that of} BFGS-LS with $T = (1, \frac{1}{4}, \frac{1}{16})$, which suggests that always trying $t = 1$ first is an excellent heuristic. These results also provide evidence of the effectiveness of performing inexact line searches, in settings where it is practical to do so. In \Cref{tab:t}, the number of steps needed until we consistently have $t_k \approx 1$ is shown.

\begin{table}
	\begin{tabular}{|l|l|c|c|c|c|}
		\hline
		\multirow{2}{*}{Data set} & \multirow{2}{*}{$n$} & \multirow{2}{*}{Identity Scaling} & \multicolumn{3}{|c|}{Number of iterations} \\
		\cline{4-6} & & & BFGS-A & BFGS-LS & BFGS-H \\ \hline
		\multirow{2}{*}{\texttt{covtype.libsvm.binary.scale}} & \multirow{2}{*}{55} & No & 797 & 57 & 62 \\
		\cline{3-6} &  & Yes & 1378 & 2 & 2 \\ \hline
		\multirow{2}{*}{\texttt{ijcnn1.tr}} & \multirow{2}{*}{23} & No & 270 & 25 & 26 \\
		\cline{3-6} &  & Yes & 369 & 3 & 2 \\ \hline
		\multirow{2}{*}{\texttt{w8a}} & \multirow{2}{*}{301} & No & 2056 & - & 289\\
		\cline{3-6} & & Yes & 2250 & 5 & 2 \\ \hline
		\multirow{2}{*}{\texttt{leu}} & \multirow{2}{*}{7130} & No & - & -  & 42 \\
		\cline{3-6} &  & Yes & 818 & 2 & 2 \\ \hline
		\multirow{2}{*}{\texttt{rcv1\_train.binary}} & \multirow{2}{*}{47237} & No & 132 & 15 & 4 \\
		\cline{3-6} &  & Yes & 205 & 3 & 4 \\ \hline
		\multirow{2}{*}{\texttt{real-sim}} & \multirow{2}{*}{20959} & No & 294 & 18 & 17 \\
		\cline{3-6} & & Yes & 490 & 2 & 2 \\ \hline
	\end{tabular}
	\caption{The number of iterations until $t_k = 1$ was consistently accepted by BFGS-LS and BFGS-H, and, for BFGS-A, the number of iterations until $t_k \geq 0.9$ for at least 80\% of the remaining iterations. A dash `-' indicates that the condition was not met before the stopping criterion was satisfied.}
	\label{tab:t}
\end{table}

Since computing $t_k$ also requires a Hessian-vector product, the cost comparison between the adaptive step size and inexact line search reverses when the algorithm nears convergence. Initially, a Hessian-vector product is faster than performing multiple backtracking iterations and repeatedly testing for the Armijo-Wolfe conditions; however, the line search (and the hybrid step selection) will eventually accept the step size $t_k = 1$ immediately, becoming essentially free, whereas computing the adaptive step size continues to require a Hessian-vector product on every step.

Curiously, L-BFGS was far more effective on the problems with large $n$ (\Cref{fig:timeL}) than on those with small $n$ (\Cref{fig:timeS}). Both LBFGS-A and LBFGS-LS were ineffective on the problems with small $n$, which suggests that the problem lies with the step directions computed by L-BFGS, rather than the step sizes. Identity scaling was also beneficial for L-BFGS on problems with large $n$, substantially reducing the convergence time in some cases. We note that we did not experiment comprehensively with varying $\ell$, the number of curvature pairs stored in L-BFGS, and used a standard choice of $\ell = \min\{\frac{n}{2},20\}$. Other choices of $\ell$ might lead to very different results on the problems in our test set.

On the other hand, identity scaling appeared to be detrimental for the BFGS-type methods on most problems, \textcolor{black}{which can be seen from the plots in \Cref{fig:timeS} and \Cref{fig:timeL} by comparing the CPU time needed for convergence. For instance, on the data set \texttt{covtype.libsvm.binary.scale}, the time to convergence for BFGS-A increased from 120s to 225s, and from 25s to 50s for BFGS-LS and BFGS-H.} In fact, identity scaling was beneficial for the BFGS\textcolor{black}{-A} method \emph{only} on the data set \texttt{leu}. The data set \texttt{leu} appears to be quite different from the other problems tested. The number of training samples for \texttt{leu} was $m = 38$, while for all other problems, $m$ was at least 20,000. Moreover, gradient descent with Armijo-Wolfe line search (GD-LS) was among the fastest methods on \texttt{leu}, while on the other test problems it was significantly outperformed by BFGS. The iteration counts shown in \Cref{tab:iterations} and \Cref{tab:t} also indicate that identity scaling worsened the performance of the BFGS methods on every problem except \texttt{leu}. Curiously, performing identity scaling led to BFGS-H accepting $t_k = 1$ at a much earlier iteration on all problems, yet the total CPU time used by BFGS-H was longer for \texttt{covtype.libsvm.binary.scale}, \texttt{rcv1.train.binary}, and \texttt{real-sim}.

GD-A was surprisingly effective, outperforming GD-LS \textcolor{black}{on every problem except for \texttt{leu}}. This is \textcolor{black}{somewhat surprising (in light of the performance of BFGS-A and BFGS-LS), and suggests that the curvature-adaptive step size may be useful for selecting hyperparameters for first-order methods}.

\section{Application to Stochastic Optimization}\label{further}

The adaptive step size can readily be extended to \emph{stochastic} optimization methods. Consider a problem of the form
\begin{equation}\label{bigERM}
L(w) = \frac{1}{N} \sum_{i=1}^N f_i(w) + h(w).
\end{equation}
If $N$ is extremely large, as is often the case in machine learning, simply evaluating $L(w)$ is an expensive operation, and line search is entirely impractical. To solve problems of the form (\ref{bigERM}), stochastic algorithms such as Stochastic Gradient Descent (SGD, \cite{SGD}) select a random subset $S_k$ of $\{f_1,\ldots,f_N\}$ at step $k$ and compute the gradient for the subsampled problem
\begin{equation}\label{subsample}
L^{(S_k)}(w) = \frac{1}{|S_k|} \sum_{f_i \in S_k} f_i(w) + h(w)
\end{equation}
as an approximation to the gradient of the loss function (\ref{bigERM}), and take a step using an empirically 
determined small and decreasing step size. In variance-reduced versions of SGD such as SVRG \cite{JZ}, it is common to use a constant step size, determined through experimentation. \textcolor{black}{The curvature-adaptive step size has two  desirable properties in this setting: it eliminates the need to select a step size through ad-hoc experimentation, and it incorporates second-order information, which is currently not exploited by most stochastic algorithms.} 


\textcolor{black}{Since the time of the initial writing of this article, related work on stochastic quasi-Newton methods has appeared in the machine learning literature. In particular, the curvature-adaptive step size was extended to stochastic gradient descent and stochastic BFGS in \cite{ZGG}. A complete discussion of stochastic optimization, and the content of \cite{ZGG}, is beyond the scope of this article. However, we have performed several preliminary experiments with stochastic versions of adaptive methods, which are presented in \Cref{apdx:stoch}, along with a summary of the relevant theory from \cite{ZGG}. These experiments (see \Cref{fig:stoch}) demonstrate that stochastic adaptive BFGS can be quite effective for solving stochastic problems.}

There is currently also little work on algorithms exploiting the finite sum structure (\ref{bigERM}), which can provably attain superlinear convergence. Aside from \cite{ZGG}, we are only aware of the Newton Incremental Method (NIM) of Rodomanov and Kropotov \cite{RK}, and the DiSCO method of Zhang and Xiao \cite{ZX}, both of which are based on Newton's method. These methods require additional memory of the order $O(N)$, which is often substantial. We are hopeful that use of the adaptive step size, and the principles behind it, will lead to new advances in the field of stochastic optimization.

\section*{Acknowledgements}
We would like to thank Jorge Nocedal for carefully reading and providing very helpful suggestions for improving an earlier version of this paper. \textcolor{black}{We also thank several anonymous referees for their comments and suggestions.}
\bibliographystyle{siam}
\raggedright	
\bibliography{adaptQN}

\appendix
\section{Stochastic Experiments}\label{apdx:stoch}
The experiments presented here are derived from the experiments in \cite[\S 4]{ZGG}. Several stochastic algorithms are tested on an \emph{online least-squares} problem of the form
$$\min_w \mathbb{E} (Y - X^Tw)^2 + \frac{1}{2}\lambda \|w\|^2.$$
\emph{Online} refers to the method of sampling: we can only access $(X,Y)$ by calling an oracle at each iteration $k$, which returns $|S_k|$ i.i.d instances of $(X,Y)$. The model for $(X,Y)$ has the following specification:
\begin{itemize}
	\item $X$ has a multivariate normal distribution $N(0, \Sigma)$, where $\Sigma$ is the covariance matrix of the \texttt{w8a} data set (see \Cref{tab:ds}).
	\item $Y = X^T\beta + \epsilon$, where $\beta$ is deterministic and sparse (80\% sparsity) and $\epsilon \sim N(0,1)$ is a noise component.
\end{itemize}

Since our model is based on the \texttt{w8a} data set, the dimension of $w$ is $p = 300$, and the regularizer is set to $\lambda = \frac{1}{p}$.

We compare the following algorithms. For a deterministic method $M$, the corresponding \emph{stochastic $M$ method} takes the step of the underlying $M$ method, but computed from the empirical objective function (\ref{subsample}) sampled at each iteration. The convergence of these methods\footnote{Note: the stochastic adaptive BFGS analyzed in \cite{ZGG} is slightly different, as it incorporates an additional Wolfe condition test.} is analyzed in \cite{ZGG}. In summary, the stochastic adaptive gradient descent method returns an $\epsilon$-optimal solution in expectation after $O(\log(\epsilon^{-1}))$ iterations when $|S_k|$ is chosen as a constant (depending on $\epsilon$), and stochastic adaptive BFGS converges $R$-superlinearly with probability 1 when $|S_k|$ increases superlinearly.
\begin{description}
	\item[SBFGS-A] The stochastic adaptive BFGS method. At each iteration, an adaptive BFGS step is computed from the empirical objective function (\ref{subsample}). The BFGS update is computed from the pair $(d_k, G_kd_k)$ which is more stable than using the pair $(s_k,y_k)$ with the difference $y_k$ of sampled gradients (see \cite{BHNS,GGR}).
	\item[SBFGS-1] The stochastic BFGS method with \emph{constant} step size $\alpha_1$. The step size $\alpha_1$ is given in \Cref{tab:stepsched}.
	\item[SN-A] Nesterov's stochastic damped Newton method \cite{N_ICP}.
	\item[SN-1] The stochastic Newton method with constant step size $\alpha_1$.
	\item[SGD-A] Stochastic adaptive gradient descent. 
	\item[SGD-$i$] Stochastic gradient descent with constant step size $\alpha_i$ for $i = 1,\ldots,4$ (\Cref{tab:stepsched}).
\end{description}

\begin{table}
	\begin{tabular}{|c|c|} \hline
		& Value \\ \hline
		$\alpha_1$ & $\frac{1}{140,000} \approx \texttt{7.14e-6}$ \\ \hline
		$\alpha_2$ & \texttt{5e-6} \\ \hline
		$\alpha_3$ & \texttt{2e-6} \\ \hline
		$\alpha_4$ & \texttt{1e-6} \\ \hline
	\end{tabular}
	\caption{Constant step sizes.}
	\label{tab:stepsched}
\end{table}

The theory \cite{ZGG} suggests taking an increasing number of samples for stochastic adaptive methods. For SBFGS-A, SBFGS-1, SN-A, SN-1, and SGD-A, we use $|S_k| = \frac{1}{2}p\cdot (1.05)^{\floor{\frac{k}{50}}}$. That is, the number of samples starts at $\frac{1}{2}p = 150$ and increases by 5\% every 50 iterations. For SGD-$i$ methods, we test three different constant batch sizes: a \emph{small} batch of $|S_k| = \frac{1}{2}p$, a \emph{medium} batch of $|S_k| = p$, and a \emph{large} batch of $|S_k| = 4p$.

\begin{figure}
	\centering
	\includegraphics[clip=true, scale=0.7]{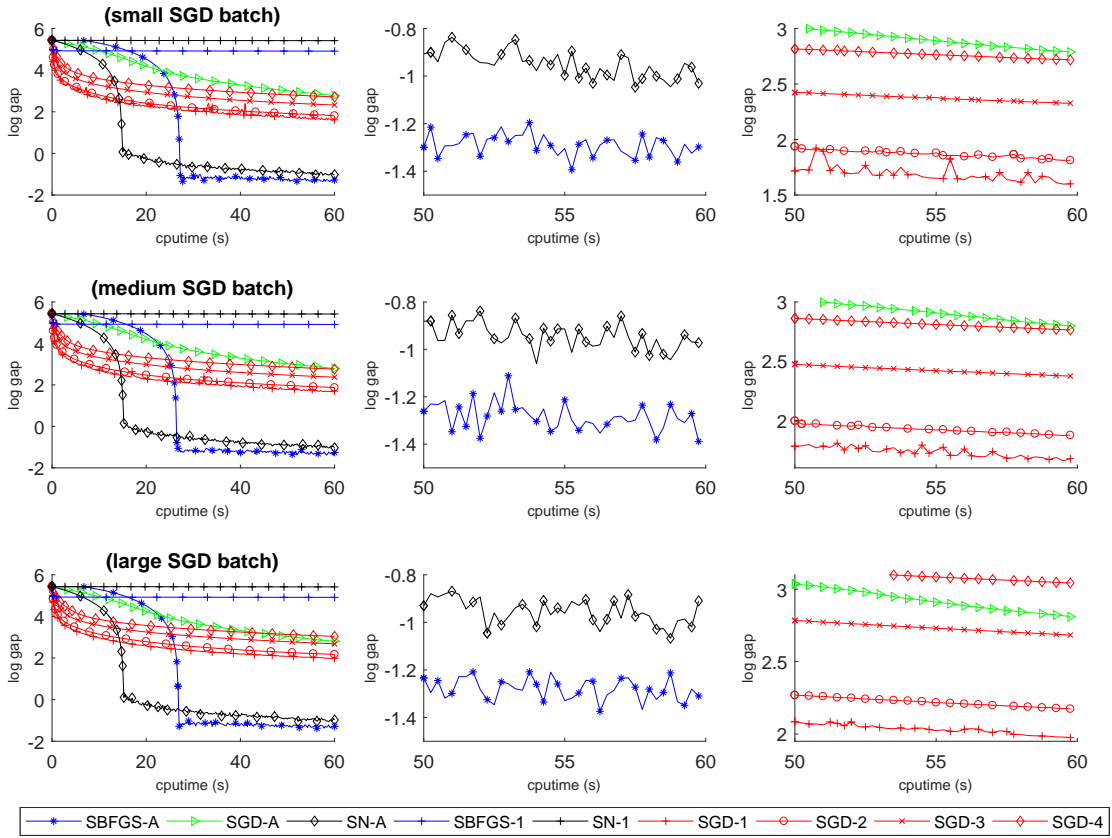}
	\caption{Performance of the stochastic algorithms. In the top row, the SGD methods SGD-1, SGD-2, SGD-3, SGD-4 use small batches ($|S_k| = \frac{1}{2}p$). Likewise, the second and third row use medium and large batches, respectively. The first column shows the performance of each method in 60s of CPU time, and the second and third columns show a close-up of the last 10s (50s-60s).}
	\label{fig:stoch}
\end{figure}

The results of the experiments are shown in \Cref{fig:stoch}. As before, the \emph{log gap} is $\log_{10}(f(x_k) - f(x_\ast))$, where $x_\ast$ is the true minimizer ($x_\ast$ can be computed explicitly given $\Sigma$ and $\beta$). The plots in the first column shows the trajectory of each method in 60 seconds of CPU time; the second and third columns show the final 10 seconds (from 50s to 60s) in greater detail. The starting point in all trials was $w = 0$.

Both SBFGS-A and SN-A exhibit superlinear convergence once they approach the minimizer. Curiously, SBFGS-A attains greater accuracy than SN-A using the same sample sizes (see second column of \Cref{fig:stoch}); we suspect that the noisiness of sampling $G_k$ damages SN-A. These methods greatly outperform SGD, even with well-tuned step sizes. We note that SGD is quite sensitive to the choice of step size. A constant step size cannot be made much larger than $\alpha_1$; using $\frac{1}{130,000} \approx \texttt{7.69e-6}$ causes SGD (even with large batches) to immediately diverge. In fact, we can check that the largest eigenvalue of $\Sigma$ is approximately \texttt{1.32e5}. Furthermore, the superior performance of SBFGS-A and SN-A depends at least partially on the curvature-adaptive step size. The methods SBFGS-1 and SN-1, which use the constant step size $\alpha_1$, converge extremely slowly\footnote{It is possible to use even larger step sizes with these methods. We observed that stochastic BFGS and stochastic Newton can tolerate much larger constant step sizes than $\alpha_1$ without diverging wildly as SGD does. However, for stochastic BFGS, the performance is not better, and is usually much worse than using $\alpha_1$, as the algorithm escapes to a worse region before beginning to decrease slowly.}, so the success of SBFGS-A and SN-A is not solely due to the second-order information in $H_k$.

SGD-A was slower than SGD with tuned step sizes. We found that the initial adaptive step size was \texttt{1e-8}, which explains the relatively slow convergence of SGD-A. It is also worth noting that even with a small initial sample, SGD-A never produced an overly large step size causing it to diverge or oscillate, something which is not strictly guaranteed by the theory.

We have not touched on the subject of variance reduction, which is generally crucial, though not particularly relevant when considering the results in \Cref{fig:stoch}. Good variance reduction techniques will be important for designing an effective, general-purpose solver based on SBFGS-A, SN-A, or indeed, any other of the stochastic algorithms tested.
\end{document}